\documentclass[12pt]{article}
\usepackage{amsmath, amssymb, amsthm,amscd}

\numberwithin{equation}{section}

\newcommand{\Z}{{\mathbb Z}}
\newcommand{\Q}{{\mathbb Q}}

\newcommand{\C}{{\mathbb C}}

\DeclareMathOperator{\End}{End} 
\DeclareMathOperator{\Mat}{Mat}

\newtheorem{lemma}{Lemma}
\newtheorem{theorem}{Theorem}
\newtheorem{proposition}[lemma]{Proposition}

\theoremstyle{definition}
\newtheorem{definition}{Definition}
\newtheorem{remark}{Remark}
\title{Finite-dimensional modules for 
the polynomial ring in one variable
as a vertex algebra}
\author{Kenichiro Tanabe\footnote{Partially supported by JSPS
Grant-in-Aid for Scientific Research No. 17740002.}\\\\
Department of Mathematics\\
Hokkaido University\\
Kita 10, Nishi 8, Kita-Ku, Sapporo, Hokkaido, 060-0810\\
Japan\\\\
ktanabe@math.sci.hokudai.ac.jp}
\date{}
\begin{document}
\maketitle
\begin{abstract}
A commutative associative algebra $A$ over $\C$ with a derivation 
is one of the simplest examples of a vertex algebra.
However, the differences between the modules for $A$ as a vertex algebra
and the modules for $A$ as an associative algebra
are not well understood. 

In this paper, I give the classification of 
finite-dimensional indecomposable untwisted or twisted modules for 
the polynomial ring in one variable
over ${\mathbb C}$ as a vertex algebra. 
\end{abstract}

\section{Introduction}
 
In \cite{B}, Borcherds gave a construction of a vertex algebra from any 
commutative associative algebra with a derivation.
Let $A$ be a commutative associative algebra with identity over $\C$ and 
$D$ a derivation of $A$, which is an endomorphism of $A$ satisfying 
$D(ab)=(Da)b+a(Db)$ for all $a,b\in A$.
For $a\in A$, we define $Y(a,x)\in(\End A)[[x]]$ by
\begin{align*}
Y(a,x)b&=\sum_{i=0}^{\infty}\dfrac{1}{i!}(D^ia)bx^{i}
\end{align*}
for $b\in A$. Then, $(A,Y,1)$ is a vertex algebra (\cite{B}, \cite[Example 3.4.6]{LL}).
Note that $(A,Y,1)$ is not a vertex operator algebra except the case that 
$A$ is finite dimensional and $D=0$ (cf. \cite[Remark 3.4.7]{LL}).
Conversely, a vertex algebra $V$ which satisfies
$Y(u,x)\in(\End V)[[x]]$ for all $u\in V$ comes from a commutative associative algebra with a derivation.

Let us consider $A$-modules as a vertex algebra. 
Let $M$ be an $A$-module as an associative algebra.
For $a\in A$, we define $Y_M(a,x)\in(\End M)[[x]]$ by
\begin{align*}
Y(a,x)u&=\sum_{i=0}^{\infty}\dfrac{1}{i!}(D^ia)ux^{i}
\end{align*}
for $u\in M$. Then, $(M,Y_M)$ is an $A$-module as a vertex algebra.
Conversely, an $A$-module $(M,Y_M)$ as a vertex algebra which satisfies
$Y_M(a,x)\in(\End M)[[x]]$ for all $a\in A$ comes from an $A$-module as an associative algebra.
However, as pointed out by Borcherds \cite{B},
there is no guarantee that every $A$-module as a vertex algebra 
comes from an $A$-module as an associative algebra.
That is, there might exists an $A$-module $(W,Y_W)$ as a vertex algebra 
which satisfies $Y_W(a,x)\not\in(\End W)[[x]]$ for some $a\in A$.
As far as I know, there is no example of a
commutative associative algebra with a derivation
whose modules as a vertex algebra are well understood. 
This is the motivation of this paper.

We investigate the modules for 
the polynomial ring $\C[t]$ with a derivation $D$ as a vertex algebra. 
We give a necessary and sufficient condition on $D$ that there exist finite dimensional
$\C[t]$-modules as a vertex algebra
which do not come from $\C[t]$-modules as an associative algebra.
For such a derivation, we give the classification of finite dimensional indecomposable $\C[t]$-modules as a vertex algebra
which do not come from $\C[t]$-modules as an associative algebra. 
Moreover, 
we obtain similar results for $g$-twisted $\C[t]$-modules
for any finite automorphism $g$ of $\C[t]$.
These classification results say that every finite dimensional indecomposable module for
the fixed point subalgebra $\C[t]^{g}=\{a\in\C[t]\ |\ ga=a\}$
is contained in some finite dimensional indecomposable untwisted or twisted $\C[t]$-module.
This result is interesting since it reminds us of
the following conjecture on vertex operator algebras: 
let $V$ be a vertex operator algebra and $G$ a finite automorphism group.
It is conjectured that under some conditions on $V$,
every irreducible module for the fixed point vertex operator subalgebra $V^G$
is contained in some irreducible $g$-twisted 
$V$-module for some $g\in G$ (cf.\cite{DVVV}).

This paper is organized as follows.
In Sect.2 we recall 
some properties 
of vertex algebras and their modules.
In Sect.3 we give the classification of 
finite dimensional indecomposable untwisted or twisted $\C[t]$-modules
as a vertex algebra
which do not come from $\C[t]$-modules as an associative algebra. 

\section{Preliminary}

We assume that the reader is familiar with the basic knowledge on
vertex algebras as presented in \cite{B,DLM1,LL}.

Throughout this paper, 
$\zeta_p$ is
a primitive $p$-th root of unity for a positive integer $p$
and $(V,Y,{\mathbf 1})$ is a vertex algebra.
Recall that $V$ is the underlying vector space, 
$Y(\cdot,x)$ is the linear map from $V$ to $(\End V)[[x,x^{-1}]]$,
and ${\mathbf 1}$ is the vacuum vector.
Let ${\mathcal D}$ be the endomorphism of $V$
defined by ${\mathcal D}v=v_{-2}{\mathbf 1}$ for $v\in V$.
Let $E_n$ denote the $n\times n$ identity matrix.
 
First, we recall some results in \cite{B} for a vertex algebra constructed from a 
commutative associative algebra with a derivation.

\begin{proposition}{\rm\cite{B}}\label{proposition:comm-alg}
The following hold:
\begin{enumerate}
\item
Let $A$ be a commutative associative algebra with identity element $1$ over $\C$ and 
$D$ a derivation of $A$.
For $a\in A$, define $Y(a,x)\in(\End A)[[x]]$ by
\begin{align*}
Y(a,x)b&=\sum_{i=0}^{\infty}\dfrac{1}{i!}(D^ia)bx^{i}
\end{align*}
for $b\in A$. Then, $(A,Y,1)$ is a vertex algebra.
\item
Let $(V,Y,{\mathbf 1})$ be a vertex algebra which satisfies
$Y(u,x)\in(\End V)[[x]]$ for all $u\in V$.
Define a multiplication on $V$ by $u v=u_{-1}v$ for $u,v\in V$. 
Then, $V$ is a commutative associative algebra with identity element ${\mathbf 1}$ and 
${\mathcal D}$ is a derivation of $V$.
\end{enumerate}
\end{proposition}
\begin{proposition}{\rm\cite{B}}\label{proposition:comm-module}
Let $A$ be a commutative associative algebra with identity element $1$  
over $\C$ and $D$ a derivation of $A$.
Let $(A,Y,1)$ be the vertex algebra constructed from 
$A$ and $D$ in Proposition \ref{proposition:comm-alg}.
\begin{enumerate}
\item
Let $M$ be an $A$-module as an associative algebra.
For $a\in A$, define $Y_M(a,x)\in(\End M)[[x]]$ by
\begin{align*}
Y(a,x)u&=\sum_{i=0}^{\infty}\dfrac{1}{i!}(D^ia)ux^{i}
\end{align*}
for $u\in M$. Then, $(M,Y_M)$ is an $A$-module as a vertex algebra.
\item
Let $(M,Y_M)$ be an $A$-module as a vertex algebra which satisfies
$Y(a,x)\in(\End M)[[x]]$ for all $a\in A$.
Define an action of $A$ on $M$ by $au=a_{-1}u$ for $a\in A$ and $u\in M$. 
Then, $M$ is an $A$-module as an associative algebra.
\end{enumerate}
\end{proposition}

For an automorphism $g$ of $V$ of finite order $p$,
set $V^r=\{u\in V\ |\ gu=\zeta_p^{r}u\}, 0\leq r\leq p-1$.
We recall the definition of $g$-twisted $V$-modules.
\begin{definition}\label{def:weak-twisted}
A $g$-twisted $V$-module $M$ is a vector space equipped with a
linear map
\begin{equation*}
Y_M(\,\cdot\,,x) : V\ni v  \mapsto Y_M(v,x) = \sum_{i \in \Z/p}
v_i x^{-i-1} \in (\End M)[[x^{1/p},x^{-1/p}]]
\end{equation*}
which satisfies the following conditions:

\begin{enumerate}
\item $Y_M(u,x) = \sum_{i \in r/p+\Z}u_i x^{-i-1}$ for $u \in V^r$.
\item $Y_M(u,x)w\in M((x^{1/p}))$ for $u \in V$ and $w \in M$.
\item $Y_M({\mathbf 1},x) = \mathrm{id}_M$.
\item For $u \in V^r$, $v \in V^{s}$, 
$m\in r/T+\Z,\ n\in s/T+\Z$, and $l\in\Z$,
\begin{align*}
&\sum_{i=0}^{\infty}\binom{m}{i}
(u_{l+i}v)_{m+n-i} \\
& = 
\sum_{i=0}^{\infty}\binom{l}{i}(-1)^i
\big(u_{l+m-i}v_{n+i}+(-1)^{l+1}v_{l+n-i}u_{m+i}\big).
\end{align*}
\end{enumerate}
\end{definition}

The following result is well known in the case of $V$-modules 
with a minor change of the conditions
(cf. \cite[Proposition 4.8]{K} and see Remark \ref{remark:weak} below).
Using the same argument as in the case of $V$-modules, one can show the case of twisted $V$-modules.
\begin{proposition}\label{proposition:iff-1}
Let $g$ be an automorphism of $V$ of finite order $p$,
$M$ a vector space, and 
$Y_{M}(\cdot,x)$ a linear map from
$V$ to $(\End M)[[x^{1/p},x^{-1/p}]]$ such that
for all $0\leq r\leq p-1$ and all $u\in V^r$,
$Y_{M}(u,x)=\sum_{i\in r/p+\Z}u_{i}x^{-i-1}$.
Then, $(M,Y_M)$ is a $g$-twisted $V$-module
if and only if the following five conditions hold:
\begin{itemize}
\item[{\rm (M1)}] For $u\in V$  and $w\in M$, $Y_M(u,x)w\in M((x^{1/p}))$.
\item[{\rm (M2)}]  $Y_M({\mathbf 1},x)= \mathrm{id}_M$.
\item[{\rm (M3)}]  For $u\in V^{r},v\in V^{s}$,  
$m\in r/p+\Z$, and $n\in s/p+\Z$, 
$[u_{m},v_{n}]=\sum_{i=0}^{\infty}\binom{m}{i}(u_{i}v)_{m+n-i}$.
\item[{\rm (M4)}]  
For $u\in V^{r}$, $v\in V^{s}$, $m\in r/p+\Z$, and $n\in s/p+\Z$, 
\begin{align*}
\sum_{i=0}^{\infty}\binom{m}{i}(u_{-1+i}v)_{m+n-i}&=
\sum_{i=0}^{\infty}(u_{-1+m-i}v_{n+i}+v_{-1+n-i}u_{m+i}).
\end{align*}
\item[{\rm (M5)}]  For $u\in V$, $Y_M({\mathcal D}u,x)=dY_M(u,x)/dx$.
\end{itemize}
\end{proposition}
\begin{remark}\label{remark:weak}
In Proposition \ref{proposition:iff-1}, 
(M4) can be replaced by the following simpler condition 
in the case of $V$-modules:
\begin{center}
{\rm (M4)$^{\prime}$}\quad
For $u,v\in V$ and $n\in \Z$, 
$(u_{-1}v)_{n}=
\sum_{i=0}^{\infty}(u_{-1-i}v_{n+i}+v_{-1+n-i}u_{i})$.
\end{center}
However, we need (M4) in the case of twisted $V$-modules.
\end{remark}
\begin{remark}\label{remark:no-twist}
Let $g$ be an automorphism of $V$ of finite order $p>1$ and 
$(M,Y_M)$ a $g$-twisted $V$-module.
Suppose that there exist $1\leq r\leq p-1$ and $u\in V^{r}$ such that $Y_M(u,x)\neq 0$.
Then, since $Y_M(u,x)=x^{-r/p}(\End M)[[x,x^{-1}]]$,
we have $Y_M(D^{k}u,x)=d^kY_M(u,x)/dx^k\not\in(\End M)[[x^{1/p}]]$
for sufficiently large integer $k$.
Therefore, there is no $g$-twisted $V$-module $(M,Y_M)$ which satisfies
$Y_M(\oplus_{r=1}^{p-1}V^r,x)\neq 0$ and $Y_M(u,x)\in (\End M)[[x^{1/p}]]$ for all $u\in V$. 
\end{remark}
\section{Finite-dimensional $\C[t]$-modules as a vertex algebra}

In this section we classify all finite-dimensional indecomposable 
untwisted or twisted $\C[t]$-modules as a vertex algebra 
which do not come from $\C[t]$-modules as an associative algebra. 
For a given derivation $D$ of $\C[t]$,
we write $(\C[t],D)$ instead of $\C[t]$ when we need to express $D$.
It is easy to see that 
any derivation $D$ of $\C[t]$ can be expressed as $D=f(t)d/dt$ where $f(t)\in\C[t]$.

Let $g$ be an automorphism of $\C[t]$ of finite order $p>1$.
Then, $g(t)=\alpha t+\beta$ where $\alpha$ is a primitive $p$-th root of unity and $\beta\in\C$.
For the automorphism $h$ of $\C[t]$ defined by $h(t)=(\alpha-1)t+\beta$,
we have $h^{-1}gh(t)=\alpha t$. Namely, 
any automorphism of $\C[t]$ of finite order conjugates to
an automorphism which preserves $\C t$.
Throughout this section, 
for any automorphism $g$ of $\C[t]$ of order $p$ which preserves $\C t$,
we always take $\C[t]^{1}=\{a\in\C[t]\ |\ ga=\xi a\}$ where $\xi$
is the primitive $p$-th root of unity determined by $gt=\xi t$.

We have the following lemma by applying 
Proposition \ref{proposition:iff-1} in the case of 
finite-dimensional untwisted or twisted $\C[t]$-modules.

\begin{lemma}\label{lemma:finite}
Let $D=f(t)d/dt, f(t)\in\C[t]$, be a derivation of $\C[t]$,
$g$ an automorphism of $\C[t]$ of finite order $p$
which preserves $\C t$,
$M$ a finite-dimensional vector space,
and $T(x)=\sum_{i\in -1/p+\Z}T_{(i)}x^{i}\in(\End M)[[x^{1/p},x^{-1/p}]]$.
Then, there exists a $g$-twisted $(\C[t],D)$-module $(M,Y_M)$ as a vertex algebra
with $Y_M(t,x)=T(x)$ if and only if the following three conditions hold:
\begin{itemize}
\item[(i)] $T(x)\in(\End M)((x^{1/p}))$.
\item[(ii)] For all $i,j\in -1/p+\Z$, $T_{(i)}T_{(j)}=T_{(j)}T_{(i)}$.
\item[(iii)] $dT(x)/dx=f(T(x))$.
\end{itemize}
In this case, for $P(t)\in\C[t]$ we have $Y_M(P(t),x)=P(Y_M(t,x))$ and hence
$(M,Y_M)$ is uniquely determined by $T(x)$.
In the case $p=1$, $(M,Y_M)$ does not come from 
a $\C[t]$-module as an associative algebra if and only if 
$T(x)\not\in(\End M)[[x]]$.
In the case $p\geq 2$, if $Y_M(t,x)\neq 0$ then
$Y_M(t,x)\not\in(\End M)[[x^{1/p}]]$.
\end{lemma}
\begin{proof}
Assume that there exists a $g$-twisted $(\C[t],D)$-module $(M,Y_M)$ as a vertex algebra
with $Y_M(t,x)=T(x)$. 
It follows form Proposition \ref{proposition:iff-1} that 
$(M,Y_M)$ satisfies (M1)--(M5) in the proposition.
Since $M$ is finite-dimensional,  we have
\begin{align}
Y_M(a,x)\in (\End M)((x^{1/p}))\label{eqn:lb}
\end{align}
for all $a\in\C[t]$ by (M1).
Let $a\in\C[t]^r$ and $b\in \C[t]^s$. Since $a_kb=0$ for all nonnegative integer $k$ in $\C[t]$
,we have
\begin{align}
[a_i,b_j]&=\sum_{k=0}^{\infty}\binom{i}{k}(a_kb)_{i+j-k}=0\label{eqn:commutative}
\end{align}
for all $i,j\in\Z/p$ by (M3). 
Therefore, we get (i) and (ii).
It follows from (\ref{eqn:lb}), (\ref{eqn:commutative}), and (M4) that
\begin{align*}
&(ab)_{i+j}=(a_{-1}b)_{i+j}=\sum_{k=0}^{\infty}\binom{i}{k}(a_{-1+k}b)_{i+j-k}\\
&=\sum_{k=0}^{\infty}(a_{-1+i-k}b_{j+k}+b_{-1+j-k}a_{i+k})
=\sum_{k\in\Z}a_{-1+i-k}b_{j+k}
\end{align*}
for $i\in r/p+\Z$ and $j\in s/p+\Z$, or equivalently $Y_M(ab,x)=Y_M(a,x)Y_M(b,x)$.
This and (M1) show that
for $P(t)\in\C[t]$,
\begin{align}
Y_M(P(t),x)&=P(Y_M(t,x)).\label{eqn:asso}
\end{align}
By (\ref{eqn:asso}) and (M5), (iii) holds since
\begin{align*}
\dfrac{d}{dx}Y_M(t,x)=Y_M(Dt,x)=Y_M(f(t),x)=f(Y_M(t,x)).
\end{align*}
In the case $p=1$, it follows from (\ref{eqn:asso}) that  $(M,Y_M)$ does not come from 
a $\C[t]$-module as an associative algebra if and only if 
$T(x)\not\in(\End M)[[x]]$.
In the case $p\geq 2$, it follows from
Remark \ref{remark:no-twist} and (\ref{eqn:asso})
that if $Y_M(t,x)\neq 0$ then
$Y_M(t,x)\not\in(\End M)[[x^{1/p}]]$ since $t\in\C[t]^1$.

Conversely,
assume that (i)--(iii) hold for $T(x)$.
For $P(t)\in\C[t]$, set $Y_M(P(t),x)=P(T(x))$.
We show that $(M,Y_M)$ satisfies (M1)--(M5) in Proposition \ref{proposition:iff-1}.
Note that $Y(t^k,x)\in x^{-k/p}(\End M)((x))$ for all nonnegative integer $k$.
Let $Y_M(a,x)=\sum_{i\in\Z/p}a_ix^{-i-1}$ and $Y_M(b,x)=\sum_{i\in\Z/p}b_ix^{-i-1}$
for $a,b\in\C[t]$.
Then, we have $Y_M(a,x)\in (\End M)((x^{1/p}))$ and
$Y_M(ab,x)=Y_M(a,x)Y_{M}(b,x)$ by the definition of $Y_M$.
We also have $[a_i,b_j]=0$ for all $i,j\in\Z/p$ by (ii).
Therefore, (M1)--(M4) hold.
By (iii),
\begin{align*}
\dfrac{d}{dx}Y_M(t,x)&=\dfrac{d}{dx}T(x)=f(T(x))\\
&=Y_M(f(t),x)=Y_M(Dt,x).
\end{align*}
Assume that for $a,b\in\C[t]$, $dY_M(a,x)/dx=Y_M(Da,x)$ and $dY_M(b,x)/dx=Y_M(Db,x)$
hold. Then,
\begin{align*}
Y_M(D(ab),x)&=Y_M((Da)b+a(Db),x)\\
&=Y_M(Da,x)Y_M(b,x)+Y_M(a,x)Y_M(Db,x)\\
&=(\dfrac{d}{dx}Y_M(a,x))Y_M(b,x)+Y_M(a,x)\dfrac{d}{dx}Y_M(b,x)\\
&=\dfrac{d}{dx}(Y_M(a,x)Y_M(b,x)).
\end{align*}
Therefore, (M5)  holds.
We conclude that $(M,Y_M)$ is a $g$-twisted $(\C[t],D)$-module as a vertex algebra.
\end{proof}
\begin{remark}For any finite automorphism $g$ of $\C[t]$,
it follows from Lemma \ref{lemma:finite} that there always exists a unique
$g$-twisted $(\C[t],D)$-module $(M,Y_M)$ with $Y_M(t,x)=0$.
By (\ref{eqn:asso}), for $P(t)=\sum_{i=0}^{m}P_it^i\in\C[t]$ 
we have $Y_M(P(t),x)=P_0$. Therefore, this is a trivial case.
\end{remark}
We introduce some notation.
Let $R$ be a commutative ring and let $\Mat_n(R)$ denote 
the set of all $n\times n$ matrices with entries in $R$. 
Let $E_{ij}$ denote the matrix whose $(i,j)$ entry is $1$ and all other entries are $0$.
Define $\Delta_{k}(R)=\{(x_{ij})\in\Mat_n(R)\ |\ x_{ij}=0\mbox{ if $i+k\neq j$}\}$ for $0\leq k\leq n$.  
Then, for $a\in \Delta_k(R)$ and $b\in \Delta_{l}(R)$, we have $ab\in \Delta_{k+l}(R)$.
For $X=(x_{ij})\in \Mat_n(R)$ and $k=0,\ldots,n-1$, 
define the matrix $X^{(k)}=\sum_{i=1}^{n}x_{i,i+k}E_{i,i+k}\in \Delta_k(R)$.
Note that for a upper triangular matrix $X$ we have $X=\sum_{k=0}^{n-1}X^{(k)}$.
Let $J_n$ denote the following $n\times n$ matrix:
\begin{align*}
J_n&=\begin{pmatrix}
0&1&0&\cdots&0\\
\vdots&\ddots&\ddots&\ddots&\vdots\\
\vdots&&\ddots&\ddots&0\\
\vdots&&&\ddots&1&\\
0&\cdots&\cdots&\cdots&0
\end{pmatrix}.
\end{align*}

For $G(x)\in x\C[[x]]$ and $m\in\Q$,
we define $(x(1+G(x)))^{m}=x^{m}\sum_{i=0}^{\infty}\binom{m}{i}G(x)^i$.
For $F(x)\in \C((x))$, $m\in\Q$, and a nilpotent $n\times n$ matrix $H$, 
we define $(E_n+F(x)H)^{m}=\sum_{i=0}^{\infty}\binom{m}{i}(F(x)H)^i$.
This sum is finite since $H^{n}=0$ and hence $(E_n+F(x)H)^{m}\in\Mat_n(\C)((x))$.

The following lemma is used to construct finite-dimensional indecomposable 
untwisted or twisted $\C[t]$-modules as a vertex algebra.
\begin{lemma}\label{lemma:nil}
Let $p$ be a positive integer, $m$ a nonzero rational number,
$r(x)$ a nonzero element of $\C((x))$, 
$F(x)\in\C((x))\setminus\C[[x]]$,  
and $H$ a nilpotent $n\times n$ matrix.
Set $U(x)=\sum_{i\in-1/p+\Z}U_{(i)}x^i=x^{-1/p}r(x)(E_n+F(x)H)^m\in\Mat_{n}(\C)((x^{1/p}))$.
Let ${\mathcal U}$ denote the subalgebra of $\Mat_{n}(\C)$
generated by all $U_{(i)}, i\in-1/p+\Z$.
Then, $H$ is an element of ${\mathcal U}$ and 
$U_{(i)}$ is a polynomial in $H$ for every $i\in-1/p+\Z$.
Moreover, $\C^n$ is an indecomposable ${\mathcal U}$-module if and only if 
$H$ conjugates to $J_n$.
\end{lemma}
\begin{proof}
We may assume that $H$ is not the zero-matrix.
Let $F(x)=\sum_{i=N}^{\infty}F_{(i)}x^i, F_{(N)}\neq 0$ with $N<0$, 
and $r(x)=\sum_{L=0}^{\infty}r_{(i)}x^i, r_{(L)}\neq 0$.
Set $\tilde{U}(x)=\sum_{i=K}^{\infty}\tilde{U}_{(i)}x^i
=(E_n+F(x)H)^m\in\Mat_{n}(\C)((x))$.
Then,
\begin{align}
U(x)&=x^{-1/p}r(x)\tilde{U}(x)=x^{-1/p}(\sum_{i=L}^{\infty}r_{(i)}x^i)
(\sum_{j=K}^{\infty}\tilde{U}_{(j)}x^j)\nonumber\\
&=x^{-1/p}\sum_{k=K+L}^{\infty}
\sum_{L\leq i, K\leq j, i+j=k}r_{(i)}\tilde{U}_{(j)}x^{k}.\label{eqn:ind-k}
\end{align}
It is easy to see from (\ref{eqn:ind-k}) that 
$\tilde{U}_{(j)}$ is in ${\mathcal U}$ for every $j\in\Z$,  and
${\mathcal U}$ is generated by $\tilde{U}_{(j)}$'s.
Since
\begin{align}
\tilde{U}(x)&=\sum_{i=0}^{d}\binom{m}{i}F(x)^iH^i\label{eqn:tildeU}
\end{align}
where $d=\max\{i\in\Z\ |\ \binom{m}{i}\neq 0\mbox{ and }H^i\neq 0\}$,
every $\tilde{U}_{(j)}, j\in\Z$, is a polynomial in $H$ and so is $U_{(j)}$.  
Since $m$ is not zero and $H$ is not the zero-matrix, $d$ is positive.
It follows from (\ref{eqn:tildeU}) that 
the term with the lowest degree of $\tilde{U}(x)$ is $\binom{m}{d}H^dx^{Nd}$
since $N$ is negative. 
Hence, $H^d$ is in ${\mathcal U}$ and 
\begin{align*}
\sum_{i=0}^{d-1}\binom{m}{i}F(x)^iH^i&=\tilde{U}(x)-\binom{m}{d}F(x)^dH^d
\end{align*}
is in ${\mathcal U}((x))$. Applying the same argument to $\sum_{i=0}^{d-1}\binom{m}{i}F(x)^iH^i$,
we have $H^{d-1}$ is in ${\mathcal U}$ and  
$\sum_{i=0}^{d-2}\binom{m}{i}F(x)^iH^i$ is in ${\mathcal U}((x))$.
Continuing in this way we get $H^k\in{\mathcal U}$ for $1\leq k\leq d$.
In particular, $H$ is an element of ${\mathcal U}$.

Since ${\mathcal U}$ is generated by $H$,
$\C^n$ is an indecomposable ${\mathcal U}$-module if and only if 
$H$ conjugates to $J_n$.
\end{proof}
For a vector space $M$ of dimension $n$, 
we identify $\End M$ with $\Mat_n(\C)$  by fixing a basis of $M$. 

Now we state our main theorems.
\begin{theorem}\label{theorem:untwist}
Let $D=f(t)d/dt, f(t)\in\C[t],$ be a derivation of $\C[t]$.
Then, there exists a finite-dimensional $(\C[t],D)$-module as a vertex algebra
which does not come from $\C[t]$-module as an associative algebra
if and only if $f\neq 0$ and $\deg f=2$.
Let $f(t)=c(t-\alpha)(t-\beta)\in\C[t]$ where $\alpha,\beta,c\in\C$ and $c\neq 0$.
Then, for every positive integer $n$, 
any $n$ dimensional indecomposable $(\C[t],D)$-module 
as a vertex algebra which does not come from a $\C[t]$-module as an associative algebra
is isomorphic to the following $(\C[t],D)$-module $(M_n,Y_{M_n})$:
\begin{enumerate}
\item In the case $\alpha=\beta$, $M_n=\C^n$ and 
\begin{align*}
Y_{M_n}(t,x)&=\alpha E_n-\dfrac{1}{c}x^{-1}(E_n-x^{-1}J_n)^{-1}.
\end{align*}
\item In the case $\alpha\neq\beta$, $M_n=\C^n$ and 
\begin{align*}
Y_{M_n}(t,x)&=\alpha E_n+(\alpha-\beta)(-1+\exp(-c(\alpha-\beta)x))^{-1}\\
&\quad\times
(E_n-\dfrac{\exp(-c(\alpha-\beta)x)}{-1+\exp(-c(\alpha-\beta)x)}J_n)^{-1}.
\end{align*}
\end{enumerate}
\end{theorem}

\begin{theorem}\label{theorem:twist}
Let $g$ be an automorphism of $\C[t]$ of finite order $p>1$
which preserves $\C t$,
and $D=f(t)d/dt, f(t)\in\C[t],$ a derivation of $\C[t]$. 
Then, there exists a finite-dimensional $g$-twisted $(\C[t],D)$-module $(M,Y_M)$ with
$Y_M(t,x)\neq 0$
if and only if $f(t)=c_1t+c_{p+1}t^{p+1}$ where $c_1, c_{p+1}\in\C$ and $c_{p+1}\neq 0$.
In this case, for every positive integer $n$, 
any $n$ dimensional indecomposable $g$-twisted  $(\C[t],D)$-module $(M,Y_M)$
which satisfies $Y_M(t,x)\neq 0$
is isomorphic to one of the following $p$ $g$-twisted $(\C[t],D)$-modules $(M^{(p,l)}_n,Y_{M^{(p,l)}_n}),\ 
0\leq l\leq p-1$:
\begin{enumerate}
\item In the case $c_1=0$, $M^{(p,l)}_n=\C^n$ and
\begin{align*}
Y_{M^{(p,l)}_n}(t,x)&=\gamma_{p}\zeta_p^{l}x^{-1/p}(E_n-x^{-1}J_n)^{-1/p}.
\end{align*}
\item
In the case $c_1\neq 0$, $M^{(p,l)}_n=\C^n$ and
\begin{align*}
Y_{M^{(p,l)}_n}(t,x)&=\gamma_{p}\zeta_p^{l}
(\dfrac{-1+\exp(-c_1px)}{-c_1p})^{-1/p}\\
&\quad \times
(E_n-\dfrac{\exp(-c_1px)}{-1+\exp(-c_1px)}J_n)^{-1/p}.
\end{align*}
\end{enumerate}
Here $\gamma_{p}$ is any fixed primitive $p$-th  root of $(-pc_{p+1})^{-1}$.
\end{theorem}
We make a remark on Theorem \ref{theorem:untwist}.
For a derivation $D$ of $\C[t]$,
Theorem \ref{theorem:untwist} says that
the existence of a $(\C[t],D)$-module which satisfies the conditions in the theorem
implies that $D=c(t-\alpha)(t-\beta)d/dt, c\neq 0$.
For the automorphism $\sigma$ of $\C[t]$ defined by $\sigma(t)=c^{-1}t+\alpha$, 
we have
$\sigma(D)=(c(\alpha-\beta)t+t^2)d/dt$, which is the same form as in
Theorem \ref{theorem:twist}.
Since $(\C[t],D)\cong (\C[t], \sigma(D))$ as vertex algebras,
Theorem \ref{theorem:untwist} corresponds to the case $p=1$ in Theorem \ref{theorem:twist}.

We shall show Theorem \ref{theorem:untwist} and Theorem \ref{theorem:twist} simultaneously.
To do this, in Theorem \ref{theorem:untwist}, 
for the derivation $D=(c_1t+c_2t^2)d/dt$
we denote the corresponding $M_n$  by $M_n^{(1,0)}$.
In the following proof, $p$ is a positive integer and 
$g=\mathrm{id}_{\C[t]}$ when we consider Theorem \ref{theorem:untwist}.
\begin{proof}
We use Lemma \ref{lemma:finite}.
In Theorem \ref{theorem:untwist},
for $D=c(t-\alpha)(t-\beta)d/dt, c\neq 0$,
we show that for every $n$, $(M_n,Y_{M_n})$ satisfies (i)--(iii) in Lemma \ref{lemma:finite}
and $Y_{M_n}(t,x)\not\in(\End M_n)[[x]]$.
Set $T(x)=\sum_{i\in\Z}T_{(i)}x^{i}=Y_{M_n}(t,x)$.
In the case $\alpha=\beta$, since
\begin{align*}
T(x)&=\alpha E_n-\dfrac{1}{c}x^{-1}(E_n-x^{-1}J_n)^{-1}\\
&=\alpha E_n-\dfrac{1}{c}\sum_{i=0}^{n-1}J_n^{i}x^{-i-1},
\end{align*}
the lowest degree of $T(x)$ is $-n$.
In the case $\alpha\neq \beta$, since
$-1+\exp(-c(\alpha-\beta)x)=\sum_{i=1}^{\infty}(-c(\alpha-\beta)x)^i/i!
\in x\C[[x]]$, the lowest degree of $T(x)$ is also $-n$.
Therefore, (i) holds
and $T(x)\not\in(\End M_n)[[x]]$.
Since $T_{(i)}$ is a polynomial in $J_n$ for each $i\in\Z$, 
(ii) holds.
It is easy to see that (iii)  holds.
Moreover, it follows from Lemma \ref{lemma:nil} that 
$M_n$ is indecomposable.
We conclude that $M_n$ is a $n$ dimensional indecomposable $(\C[t],D)$-module
as a vertex algebra
which does not come from a $\C[t]$-module as an associative algebra.
In Theorem \ref{theorem:twist},
for $D=(c_1t+c_{p+1}t^{p+1})d/dt, c_{p+1}\neq 0$,
the same argument shows that 
$(M^{(p,l)}_n,Y_{M^{(p,l)}_n}), 0\leq l\leq p-1,$ are 
$n$ dimensional indecomposable $g$-twisted $(\C[t],D)$-modules
with $Y_{M_n^{(p,l)}}(t,x)\neq 0$.

Conversely, assume that there exists a finite-dimensional 
$g$-twisted $(\C[t],D)$-module $(M,Y_M)$ as a vertex algebra
with $Y_M(t,x)\neq 0$
which does not come from a $\C[t]$-module as an associative algebra.
We may assume that $M$ is indecomposable.

Assume that $D=0$. 
It follows from (M5) in Proposition \ref{proposition:iff-1} that
$Y_M(a,x)=a_{-1}\in(\End M)[[x]]$ for $a\in \C[t]$.
In Theorem \ref{theorem:untwist} this implies that $M$ comes from a $\C[t]$-module as an 
associative algebra by Proposition \ref{proposition:comm-module}
and in Theorem \ref{theorem:twist} this implies $Y_M(t,x)=0$.
This is a contradiction and hence $D\neq 0$.

Let $f(x)=\sum_{i=0}^{m+1}c_ix^i\in\C[t]$, $c_{m+1}\neq 0$.
We denote $Y_M(t,x)$ by $T(x)=\sum_{i\in -1/p+\Z}T_{(i)}x^{i}$.
By (i) in Lemma \ref{lemma:finite}, $T(x)$ is in $(\End M)((x^{1/p}))$.
By (ii) in Lemma \ref{lemma:finite},  we may assume that $T(x)$ is a upper triangular matrix.  
Let 
\begin{align*}
T(x)^{(k)}=\sum_{i\in -1/p+\Z}T_{(i)}^{(k)}x^{i}\in 
\Delta_{k}((\End M)((x^{1/p})))
\end{align*} for $0\leq k\leq n-1$.
For every $j\in -1/p+\Z$,
since $T_{(j)}^{(0)}$ is the semisimple part of $T_{(j)}$ 
and $M$ is indecomposable,
$T^{(0)}_{(j)}$ is a scalar matrix.
Let $T(x)^{(0)}=s(x)E_n$ where $s(x)=\sum_{i\in -1/p+\Z}s_{(i)}x^{i}\in\C((x^{1/p}))$.
By (iii) in Lemma \ref{lemma:finite}, we have
\begin{align*}
\dfrac{d}{dx}T(x)&=f(T(x))\\
&=\sum_{i=0}^{m+1}c_i(T(x)^{(0)}+\cdots+T(x)^{(n-1)})^i\\
&=\sum_{k=0}^{n-1}\sum_{i=0}^{m+1}c_i
\sum_{j_1+\cdots+j_i=k}T(x)^{(j_1)}\cdots T(x)^{(j_i)}.
\end{align*} 
Since $\sum_{i=0}^{m+1}c_i
\sum_{j_1+\cdots+j_i=k}T(x)^{(j_1)}\cdots T(x)^{(j_i)}\in\Delta_k((\End M)((x^{1/p})))$,
we have 
\begin{align}
\dfrac{d}{dx}s(x)E_n&=\dfrac{d}{dx}T(x)^{(0)}
=\sum_{i=0}^{m+1}c_i(T(x)^{(0)})^i=f(s(x))E_n
\label{eqn:s}
\end{align}
and 
\begin{align*}
&\dfrac{d}{dx}T(x)^{(k)}
=\sum_{i=0}^{m+1}c_i
\sum_{j_1+\cdots+j_i=k}T(x)^{(j_1)}\cdots T(x)^{(j_i)}\\
&=\sum_{i=0}^{m+1}ic_is(x)^{i-1}T(x)^{(k)}+
\sum_{i=0}^{m+1}c_i\sum_{\underset
{\scriptstyle j_1+\cdots+j_i=k}{\scriptstyle 0\leq j_1,\ldots,j_i<k}}T(x)^{(j_1)}\cdots T(x)^{(j_i)}
\end{align*} 
for $1\leq k\leq n-1$, or equivalently
\begin{align}
&\big(\dfrac{d}{dx}-\sum_{i=1}^{m+1}ic_is(x)^{i-1}\big)T(x)^{(k)}\nonumber\\
&=
\sum_{i=0}^{m+1}c_i\sum_{\underset
{\scriptstyle j_1+\cdots+j_i=k}{\scriptstyle 0\leq j_1,\ldots,j_i<k}}T(x)^{(j_1)}\cdots T(x)^{(j_i)}.
\label{eqn:induc}
\end{align} 

We show by induction on $k$ that if $s(x)\in\C[[x^{1/p}]]$, then
$T(x)^{(k)}\in (\End M)[[x^{1/p}]]$.
The case $k=0$ follows from $T(x)^{(0)}=s(x)E_n$.
For $k>0$, suppose that the lowest degree of $T(x)^{(k)}$ is negative. Then,
the lowest degree of the left-hand side of (\ref{eqn:induc}) is also negative.
This is a contradiction since the right-hand side of (\ref{eqn:induc}) is
in $(\End M)[[x^{1/p}]]$ by the induction assumption. 
Therefore, the condition $T(x)\not\in(\End M)[[x^{1/p}]]$ implies $s(x)\not\in\C[[x^{1/p}]]$,
or equivalently the lowest degree $L$ of $s(x)$ is negative. 

By (\ref{eqn:s}),
$s(x)$ satisfies
\begin{align}
\sum_{L-1\leq i\in -1/p+\Z}(i+1)s_{(i+1)}x^{i}
&=\sum_{j=0}^{m+1}c_j
\big(\sum_{L\leq i\in -1/p+\Z}s_{(i)}x^{i}\big)^j.\label{eqn:expand-s}
\end{align}
In (\ref{eqn:expand-s}), the term with the lowest degree of the left-hand side is
$Ls_{(L)}x^{L-1}$ and 
the term with the lowest degree of the right-hand side is
$c_{m+1}s_{(L)}^{m+1}x^{L(m+1)}$.
Comparing these terms, we have 
$mL=-1$ and $L=c_{m+1}s_{(L)}^{m}$. Since $L\in -1/p+\Z$, we have
$m=p, L=-1/p$, and $s_{(-1/p)}=\gamma_{p}\zeta_p^{l}, 0\leq l\leq p-1$. 
Therefore, $\deg f=2$ if $p=1$.
Let us consider the case $p\geq 2$.
The right-hand side of (\ref{eqn:expand-s}) is expanded in the following 
form:
\begin{align*}
c_{p+1}s_{-1/p}^{p+1}x^{-1-1/p}+c_{p}s_{-1/p}^px^{-1}+\cdots+
c_{1}s_{-1/p}x^{-1/p}+c_{0}x^{0}+\cdots.
\end{align*}
Since the left-hand side of (\ref{eqn:expand-s}) is in $x^{-1/p}(\End M)((x))$
and $p\geq 2$, we have
$c_j=0$ for all $j\neq 1,p+1$ and hence $f(t)=c_1t+c_{p+1}t^{p+1}$.

We may assume that $f(t)=c_1t+c_{p+1}t^{p+1}, c_{p+1}\neq 0$ by 
the remark before giving the proof.

First, we shall treat the case $c_1=0$.
In this case (\ref{eqn:s}) becomes
\begin{align}
\dfrac{d}{dx}s(x)&=c_{p+1}s(x)^{p+1}.\label{eqn:diff-s-1}
\end{align}
Solving the differential equation (\ref{eqn:diff-s-1}) in $\C((x^{1/p}))$
under the condition that $s(x)\not\in\C[[x^{1/p}]]$,
we have $s(x)=\gamma_{p}\zeta_{p}^{l}x^{-1/p}, 0\leq l\leq p-1$.
Hence
(\ref{eqn:induc}) becomes
\begin{align}
&\big(\dfrac{d}{dx}+\frac{p+1}{p}x^{-1}\big)T(x)^{(k)}
=\sum_{j\in -1/p+\Z}(j+\dfrac{p+1}{p})T_{(j)}^{(k)}x^{j-1}\nonumber\\
&=
c_{p+1}\sum_{\underset
{\scriptstyle j_1+\cdots+j_{p+1}=k}{\scriptstyle 0\leq j_1,\ldots,j_{p+1}<k}}T(x)^{(j_1)}\cdots T(x)^{(j_{p+1})}.
\label{eqn:one}
\end{align} 
Set $H=T_{(-1-1/p)}^{(1)}+\cdots+T_{(-1-1/p)}^{(n-1)}$,
which is the nilpotent part of $T_{(-1-1/p)}$. 
We shall prove that $T(x)^{(k)}$ is uniquely determined by $s(x)$ and $H$
by induction on $k$.
The case $k=0$ follows from $T(x)^{(0)}=s(x)E_n$.
For $k>0$, $T_{(-1-1/p)}^{(k)}$ is given and the other
$T(x)^{(k)}_{(j)}$ are determined from (\ref{eqn:one}) by induction assumption.

Set 
\begin{align*}
T_{H,l}(x)&=\gamma_{p}\zeta_p^{l}x^{-1/p}
(E_n-x^{-1}\dfrac{p}{\gamma_{p}\zeta_p^{l}}H)^{-1/p}
=\sum_{i\in-1/p+\Z}(T_{H,l})_{(i)}x^{i}.
\end{align*}
Using the same argument in the first part of the proof,
one can show that $T_{H,l}(x)$ satisfies (i)--(iii) in Lemma \ref{lemma:finite},
$T_{H,l}(x)^{(0)}=s(x)$, and 
$(T_{H,l})_{(-1-1/p)}=H$.
Therefore, $T(x)=T_{H,l}(x)$.
It follows from Lemma \ref{lemma:nil} that
$p(\gamma_{p}\zeta_p^{l})^{-1}H$ conjugates to $J_n$ since $M$ is indecomposable.
We conclude that $M$ is isomorphic to $M_n^{(p,l)}$.

Next, we shall treat the case $c_1\neq 0$.
In this case (\ref{eqn:s}) becomes
\begin{align}
\dfrac{d}{dx}s(x)&=c_1s(x)+c_{p+1}s(x)^{p+1}.\label{eqn:diff-s-2}
\end{align}
Solving the differential equation (\ref{eqn:diff-s-2}) in $\C((x^{1/p}))$
under the condition that $s(x)\not\in\C[[x^{1/p}]]$,
we have 
\begin{align*}
s(x)&=\gamma_{p}\zeta_p^{l}(\dfrac{-1+\exp(-c_1px)}{-c_1p})^{-1/p}, 0\leq l\leq p-1.
\end{align*}
Define $y=(-1+\exp(-c_1px))/(-c_1p)\in x\C[[x]]$ and
$\tilde{T}(y)=\gamma_{p}^{-1}\zeta_p^{-l}y^{1/p}T(-\log(1-c_1py)/(c_1p))\in\C((y))$.
Note that $\tilde{T}(y)^{(0)}=E_n$.
Then,
\begin{align*}
\dfrac{dy}{dx}&=\exp(-c_1px)=1-c_1py,\\
\dfrac{dT(x)}{dx}&=\dfrac{dy}{dx}\dfrac{d}{dy}(\gamma_{p}\zeta_p^{l} y^{-1/p}\tilde{T}(y))\\
&=\gamma_{p}\zeta_p^{l} y^{-1/p}\big(
(1-c_1py)\dfrac{d\tilde{T}(y)}{dy}
+\dfrac{-1}{p}y^{-1}\tilde{T}(y)+c_1\tilde{T}(y)\big).
\end{align*}
Therefore, by (iii) in Lemma \ref{lemma:finite} we have
\begin{align}
py(1-c_1py)\dfrac{d\tilde{T}(y)}{dy}&=
\tilde{T}(y)-\tilde{T}(y)^{p+1}.\label{eqn:s-y}
\end{align}
Define $z=-c_{1}py/(1-c_{1}py)\in y\C[[y]]$ and 
$\hat{T}(z)=\sum_{j\in\Z}\hat{T}_{(j)}z^j=\tilde{T}(z/(-c_1p(1-z)))
\in (\End M)((z))$.
Then,
\begin{align*}
\dfrac{dz}{dy}&=\dfrac{-c_1p}{(1-c_1py)^2}=-c_1p(1-z)^2,\\
y(1-c_1py)\dfrac{d\tilde{T}(y)}{dy}&
=\dfrac{-z}{c_{1}p(1-z)^2}\dfrac{dz}{dy}\dfrac{d\hat{T}(z)}{dz}=z\dfrac{d\hat{T}(z)}{dz}.
\end{align*}
By (\ref{eqn:s-y}) we have
\begin{align*}
pz\dfrac{d}{dz}\hat{T}(z)&=\hat{T}(z)-\hat{T}(z)^{p+1}.
\end{align*}
and hence
\begin{align}
p(z\dfrac{d}{dz}+1)\hat{T}(x)^{(k)}
&=-\sum_{\underset
{\scriptstyle j_1+\cdots+j_{p+1}=k}{\scriptstyle 0\leq j_1,\ldots,j_{p+1}<k}}\hat{T}(x)^{(j_1)}\cdots \hat{T}(x)^{(j_{p+1})}.
\label{eqn:induc-z}
\end{align}
for $1\leq k\leq n-1$ since $\hat{T}(z)^{(0)}=E_n$.
Using the same argument as in the case $c_1=0$, 
one can show that 
$\hat{T}(z)$ is uniquely determined by
$H=\hat{T}_{(-1)}^{(1)}+\cdots+\hat{T}_{(-1)}^{(n-1)}$,
which is the nilpotent part of $\hat{T}_{(-1)}$,
under the condition $\hat{T}(z)^{(0)}=E_n$ and that
\begin{align*}
\hat{T}(z)&=(E_n-pH z^{-1})^{-1/p}.
\end{align*}
Since $z=(-1+\exp(-c_1px))/\exp(-c_1px)$ and $\hat{T}(z)=s(x)^{-1}T(x)$,
we have 
\begin{align*}
T(x)&=\gamma_{p}\zeta_p^{l}
(\dfrac{-1+\exp(-c_1px)}{-c_1p})^{-1/p}\\
&\quad \times
(E_n-pH\dfrac{\exp(-c_1px)}{-1+\exp(-c_1px)})^{-1/p}.
\end{align*}
It follows from Lemma \ref{lemma:nil} that
$pH$ conjugates to $J_n$ since $M$ is indecomposable.
We conclude that $M$ is isomorphic to $M_n^{(p,l)}$.
\end{proof}

\begin{remark}
Let $g$ be an automorphism of $\C[t]$ of finite order $p$ which preserves $\C t$ and 
$D=(c_1t+c_{p+1}t^{p+1})d/dt$ with $c_{p+1}\neq 0$.
It is easy to see that $D$ is invariant under the action of $g$ and the fixed point subalgebra 
$\C[t]^{g}=\{a\in\C[t]\ |\ ga=a\}$ of $\C[t]$ is equal to $\C[t^p]$. 
Hence, $\C[t^p]$ is invariant under the action of $D$ and 
$(\C[t^{p}],D)$ is a subalgebra of $(\C[t],D)$.
Moreover,
$(\C[t^{p}],D)\cong (\C[t],\tilde{D})$ as associative algebras over $\C$ 
with derivations where $\tilde{D}=(pc_1t+pc_{p+1}t^2)d/dt$
and hence $(\C[t^{p}],D)\cong (\C[t],\tilde{D})$ as vertex algebras.
Therefore, we have all finite-dimensional indecomposable $(\C[t^{p}],D)$-modules
as a vertex algebra by Theorem \ref{theorem:untwist}. 
We discuss a relation between the finite dimensional $(\C[t^{p}],D)$-modules 
and finite dimensional untwisted or $g$-twisted $(\C[t],D)$-modules as vertex algebras.

By Theorem \ref{theorem:twist} and (\ref{eqn:asso}), we have
\begin{align*}
Y_{M^{(p,l)}_n}(t^{p},x)&=\big(\gamma_{p}\zeta_p^{l}x^{-1/p}(E_n-x^{-1}J_n)^{-1/p}\big)^{p}\\
&=(-pc_{p+1})^{-1}x^{-1}(E_n-x^{-1}J_n)^{-1}
\end{align*}
if $c_1=0$ and
\begin{align*}
&Y_{M^{(p,l)}_n}(t^p,x)\\
&=
\big(\gamma_{p}\zeta_p^{l}
(\dfrac{-1+\exp(-c_1px)}{-c_{1}p})^{-1/p}
(E_n-\dfrac{\exp(-c_1px)}{-1+\exp(-c_1px)}J_n)^{-1/p}\big)^{p}\\
&=\dfrac{-c_1}{c_{p+1}}
(1-\exp(-c_1px))^{-1}
(E_n-\dfrac{\exp(-c_1px)}{-1+\exp(-c_1px)}J_n)^{-1}
\end{align*}
if $c_1\neq 0$.
Therefore, it follows from Theorem \ref{theorem:untwist}
that every finite dimensional indecomposable $(\C[t^p],D)$-module
as a vertex algebra
which does not come from a $\C[t^p]$-module as an associative algebra
is obtained from finite dimensional indecomposable $g$-twisted $(\C[t],D)$-modules.

Moreover,
it is easy to see that
every finite dimensional indecomposable $(\C[t^p],D)$-module 
as an associative algebra
is contained in some finite dimensional indecomposable $(\C[t],D)$-module
as an associative algebra.
Therefore, every finite dimensional indecomposable $(\C[t^p],D)$-module
as a vertex algebra
is contained in some finite 
dimensional indecomposable untwisted or $g$-twisted $(\C[t],D)$-module as a vertex algebra.

\end{remark}

\end{document}